\documentclass[reqno]{amsart}
\usepackage{amsmath,amssymb}
\usepackage{color,graphicx}
\usepackage[latin1]{inputenc}
\usepackage{times}
\usepackage{float}
\usepackage[dvips]{hyperref}
\usepackage{amscd}
\usepackage[all]{xy}
\usepackage{color}

\DeclareMathOperator {\id}{id}

\DeclareMathOperator {\Div}{Div}

\DeclareMathOperator {\Mat}{Mat}
\DeclareMathOperator {\rank}{rank}

\newcommand {\RR}{{\mathbb R}}
\newcommand {\ZZ}{{\mathbb Z}}
\newcommand {\TT}{{\mathbb T}}

\newcommand {\calO}{{\mathcal O}}
\newcommand {\calD}{{\mathcal D}}
\newcommand {\calK}{{\mathcal K}}

\newcommand {\arxiv}[3]{#1, \emph{#2}, preprint \href{http://arxiv.org/abs/#3}{arxiv:#3}.}

\newtheorem {theorem}{Theorem}[section]

\newtheorem {lemma}[theorem]{Lemma}

\newtheorem {corollary}[theorem]{Corollary}
\theoremstyle {definition}
\newtheorem {definition}[theorem]{Definition}
\newtheorem {example}[theorem]{Example}

\newtheorem {notation}[theorem]{Notation}
\theoremstyle {remark}
\newtheorem {remark}[theorem]{Remark}

\begin{document}

\title {Chern classes of tropical vector bundles}
\author {Lars Allermann}
\address {Lars Allermann, Fachbereich Mathematik, TU Kaiserslautern, Postfach 3049, 67653 Kaiserslautern, Germany}
\email {allerman@mathematik.uni-kl.de}

\begin{abstract}
We introduce tropical vector bundles, morphisms and rational
sections of these bundles and define the pull-back of a tropical
vector bundle and of a rational section along a morphism. Most of
the definitions presented here for tropical vector bundles will be
contained in \cite{T09} for the case of line bundles. Afterwards
we use the bounded rational sections of a tropical vector bundle
to define the Chern classes of this bundle and prove some basic
properties of Chern classes. Finally we give a complete
classification of all vector bundles on an elliptic curve up to
isomorphisms.
\end{abstract}

\maketitle

\section{Tropical vector bundles} \label{sec-vectorbundles}

In this section we will introduce our basic objects such as
tropical vector bundles, morphisms of tropical vector bundles and
rational sections.

\begin{definition}[Tropical matrices]
  A tropical matrix is an ordinary matrix with entries in the tropical semi-ring
  $$(\TT=\RR \cup \{-\infty\},\oplus,\odot),$$
  where $a \oplus b = \max\{a,b\}$ and $a \odot b = a+b$.
  We denote by $\Mat(m \times n, \TT)$ the set of
  tropical $m \times n$ matrices. Let $A \in \Mat(m \times n,
  \TT)$ and $B \in \Mat(n \times p, \TT)$. We can form a tropical
  matrix product $A \odot B := (c_{ij}) \in \Mat(m \times p,\TT)$
  where $c_{ij}=\bigoplus_{k=1}^m a_{ik} \odot b_{kj}.$ Moreover,
  let $G(r \times s) \subseteq \Mat(r \times s, \TT)$ be the subset of
  tropical matrices with at most one finite entry in every row. Let
  $G(r)$ be the subset of $G(r \times r)$ containing all tropical
  matrices with exactly one finite entry in every row and every column.
\end{definition}

\begin{remark} \label{remark-map_f_A}
  Note that a matrix $A \in G(r \times s)$ does, in general, not induce a map
  $f_A: \RR^s \rightarrow \RR^r: {x \mapsto A \odot x}$ as the
  vector $A \odot x$ may contain entries that are $-\infty$. To obtain a
  map $f_A: \RR^s \rightarrow \RR^r$ anyway we use the following definition:
  Let $x \in \RR^s$ and $A \odot x = (y_1,\ldots,y_r) \in \TT^r$ with
  $y_i=-\infty$ for $i \in I$ and $y_i \in \RR$ for $i \not\in I$.
  Then we define $f_A(x) := (\widetilde{y_1},\ldots,\widetilde{y_r})
  \in \RR^r$ with $\widetilde{y_i}:=0$ for $i \in I$ and
  $\widetilde{y_i}:=y_i$ for $i \not\in I$.
\end{remark}

\begin{notation}
  For an element $\sigma$ of the symmetric group $S_r$ we denote by $A_\sigma$
  the tropical matrix $A_\sigma = (a_{ij}) \in \Mat(r \times r, \TT)$ given by $$a_{ij} := \left\{ \begin{array}{rl}
  0, & \text{ if } j=\sigma(i) \\ -\infty, & \text{ else.}\end{array} \right.$$
  Moreover, for $a_1,\ldots,a_r \in \RR$ we denote by
  $D(a_1,\ldots,a_r)$ the tropical diagonal matrix $D(a_1,\ldots,a_r)=(d_{ij}) \in \Mat(r \times r, \TT)$
  given by $$d_{ij} := \left\{ \begin{array}{rl}
  a_i, & \text{ if } i=j \\ -\infty, & \text{ else.}\end{array} \right.$$
  Note that every element $M \in G(r)$ can be written as $M=A_\sigma \odot D(a_1,\ldots,a_r)$
  for some $\sigma \in S_r$ and some numbers $a_1,\ldots,a_r \in
  \RR$. Moreover, $G(r)$ together with tropical
  matrix multiplication is a group with neutral element
  $E:=D(0,\ldots,0)$.
\end{notation}

\begin{lemma}
  $G(r)$ is precisely the set of invertible tropical matrices, i.e.
  $$G(r)=\{A \in \Mat(r \times r, \TT)| \exists A' \in \Mat(r \times r,
  \TT): A \odot A'=A' \odot A=E\}.$$
\end{lemma}
\begin{proof}
  The inclusion $$G(r) \subseteq \{A \in \Mat(r \times r, \TT)| \exists A' \in \Mat(r \times r,
  \TT): A \odot A'=A' \odot A=E\}$$ is obvious. Thus, let $A, A' \in \Mat(r \times r,
  \TT)$ be given such that $A \odot A'=A' \odot A=E$. Assume that
  $A=(a_{ij})$ contains more than one finite entry in a row or column. For simplicity of notation
  we assume that $a_{11}, a_{12} \neq -\infty$. As $A \odot A' =
  E$ we can conclude that the first two rows of $A'$ look as
  follows:
  $$A' = \left( \begin{array}{c}
    \begin{array}{cccc} \alpha & -\infty & \dots & -\infty \end{array} \\
    \begin{array}{cccc} \beta & -\infty & \dots & -\infty \end{array} \\
    \hline \\
    \text{\Huge{*}}
  \end{array} \right) \text{ for some } \alpha, \beta \in \RR.$$
  As moreover $A' \odot A =E$ holds, we can conclude from the second
  line of $A'$ and the first column of $A$ that
  $$a_{11} + \beta = -\infty,$$
  which is a contradiction to $a_{11}, \beta \in \RR$.
\end{proof}

We have all requirements now to state our main definition:

\begin{definition}[Tropical vector bundles] \label{def-vectorbundle}
  Let $X$ be a tropical cycle (cf. \cite[definition 5.12]{AR07}). A \emph{tropical vector bundle} over
  $X$ of rank $r$ is a tropical cycle $F$ together with a morphism $\pi: F
  \rightarrow X$ (cf. \cite[definition 7.1]{AR07}) and a finite open covering $\{U_1,\ldots,U_s\}$ of $X$ as well as
  a homeomorphism $\Phi_i: \pi^{-1}(U_i) \stackrel{\cong}{\rightarrow} U_i \times
  \RR^r$ for every $i \in \{1,\ldots,s\}$
  such that
  \begin{enumerate}
    \item for all $i$ we obtain a commutative diagram
        $$\begin{xy}
          \xymatrix{
            \pi^{-1}(U_i) \ar[r]^{\Phi_i} \ar[rd]_\pi & U_i \times \RR^r \ar[d]^{p_1}\\
            & U_i
          }
        \end{xy}$$
        where $p_1:U_i \times \RR^r \rightarrow U_i$ is the
        projection to the first factor,
      \item for all $i,j$ the composition $p_j^{(i)} \circ
        \Phi_i:\pi^{-1}(U_i)\rightarrow \RR$ is a regular invertible
        function (cf. \cite[definition 6.1]{AR07}), where $p_j^{(i)}: U_i \times \RR^r \rightarrow \RR:
        (x,(a_1,\ldots,a_r)) \mapsto a_j$,
    \item for every $i,j \in \{1,\ldots,s\}$ there exists a
    \emph{transition map} $M_{ij}: U_i \cap U_j \rightarrow G(r)$ such that $$\Phi_j \circ \Phi_i^{-1}: (U_i \cap
    U_j) \times \RR^r \rightarrow (U_i \cap U_j) \times \RR^r$$ is given by $(x,a)
    \mapsto (x, M_{ij}(x) \odot a)$ and the entries of $M_{ij}$ are
    regular invertible functions on $U_i \cap U_j$ or constantly $-\infty$,
    \item there exist representatives $F_0$ of $F$ and $X_0$ of $X$ such that
    $F_0= \{ \pi^{-1}(\tau)| \tau \in X_0 \}$ and
    $\omega_{F_0}(\pi^{-1}(\tau))=\omega_{X_0}(\tau)$ for all maximal polyhedra $\tau \in X_0$.
  \end{enumerate}
  An open set $U_i$ together with the map $\Phi_i: \pi^{-1}(U_i) \stackrel{\cong}{\rightarrow} U_i \times \RR^r$
  is called a \emph{local trivialization} of $F$. Tropical vector bundles of rank one are called \emph{tropical line bundles}.
\end{definition}

\begin{remark}
  Let $V_1,\ldots,V_t$ be any open covering of $X$. Then the
  covering $\{U_i \cap V_j\}$ together with the restricted
  homeomorphisms $\Phi_i|_{\pi^{-1}(U_i \cap V_j)}$ and transition
  maps $M_{ij}|_{(U_i \cap V_k) \cap (U_j \cap V_l)}$ fulfills all
  requirements of definition \ref{def-vectorbundle}, too, and hence
  defines again a vector bundle. As the open covering, the
  homeomorphisms and the transition maps are part of the data of
  definition \ref{def-vectorbundle} this new bundle is (according to our definition)
  different from our initial one even though they are ``the same'' in
  some sense. Hence, in the following we will identify vector bundles that arise by
  such a construction one from the other:
\end{remark}

\begin{definition} \label{def-identifyequalbundles}
  Let $\pi: F \rightarrow X$ together with open covering
  $U_1,\ldots,U_s$, homeomorphisms $\Phi_i$ and transition maps
  $M_{ij}$ and $\pi: F \rightarrow X$ together with open covering
  $V_1,\ldots,V_t$, homeomorphisms $\Psi_i$ and transition maps
  $N_{ij}$ be two tropical vector bundles according to definition
  \ref{def-vectorbundle}. We will identify these vector bundles
  if the vector bundles $\pi: F \rightarrow X$ with open covering
  $\{U_i \cap V_j\}$ and restricted homeomorphisms $\Phi_i|_{\pi^{-1}(U_i \cap V_j)}$
  respectively $\Psi_j|_{\pi^{-1}(U_i \cap V_j)}$ and transition maps
  $M_{ij}|_{(U_i \cap V_k) \cap (U_j \cap V_l)}$ respectively
  $N_{kl}|_{(U_i \cap V_k) \cap (U_j \cap V_l)}$ are equal.
\end{definition}

\begin{remark} \label{remark-refinedcovering}
  Let $\pi_1: F_1 \rightarrow X$ and $\pi_2: F_2 \rightarrow X$ be two vector bundles on
  $X$. By definition \ref{def-identifyequalbundles} we can always assume that $F_1$
  and $F_2$ satisfy definition \ref{def-vectorbundle} with the same open
  covering.
\end{remark}

\begin{remark}
  Let $\pi: F \rightarrow X$ be a vector bundle with open covering
  $U_1,\ldots,U_s$ and transition maps $M_{ij}$ as in definition \ref{def-vectorbundle}. On the common intersection
  $U_i \cap U_j \cap U_k$ we obviously have $M_{ij}(x) = M_{kj}(x) \odot
  M_{ik}(x)$. This last equation is called \emph{cocycle condition}.
  Conversely, given an open covering $U_1,\ldots,U_s$ of $X$
  and maps $M_{ij}: U_i \cap U_j \rightarrow G(r)$
  such that the entries of $M_{ij}(x)$ are regular invertible functions
  on $U_i \cap U_j$ or constantly $-\infty$ and the cocycle condition $M_{ij}(x) = M_{kj}(x) \odot M_{ik}(x)$ holds on $U_i \cap U_j \cap
  U_k$, we can construct a vector bundle $\pi: F \rightarrow X$ with this given open covering
  and transition functions $M_{ij}$: Take the disjoint union $\coprod_{i=1}^s (U_i \times
  \RR^r)$ and identify points $(x,y) \sim (x,M_{ij}(x) \odot a)$
  to obtain the topological space $|F|$. We have to equip this
  space with the structure of a tropical cycle. As this construction is exactly the same as for tropical line bundles,
  we only sketch it here and refer to \cite{T09} for more details. Let
  $(((X_0,|X_0|,\{ \varphi_\sigma \}),\omega_{X_0}), \{\Phi_\sigma \})$ be a representative of $X$. We define
  $F_0:= \{\pi^{-1}(\sigma)| \sigma \in X_0\}$ and
  $\omega_{F_0}(\pi^{-1}(\sigma)):=\omega_{X_0}(\sigma)$ for all
  maximal polyhedra $\sigma \in X_0$. Our next step is to
  construct the polyhedral charts $\widetilde{\varphi}_{\pi^{-1}(\sigma)}$ for $F_0$:
  Let $\sigma \in X_0$ be given and let $U_{i_1},\ldots,U_{i_t}$ be all open
  sets with non-empty intersection with $\sigma$. Moreover, let
  $\{V_i | i \in I\}$ be the set of all connected components of all $\sigma
  \cap U_{i_k}$. Every such set $V_i$ comes from a set $U_{j(i)}$ of the given open
  covering. Hence, for every pair $k,l \in I$ we have a restricted transition map
  $N_{kl}:=M_{j(k),j(l)}|_{V_k \cap V_l}$. This implies that for all $k,l \in I$
  the entries of $N_{kl} \circ \Phi_\sigma^{-1}$ are (globally) integer affine linear functions
  on $V_k \cap V_l$. As $\sigma$ is simply connected, for every such entry $h \in \calO^{*}(V_k \cap V_l)$
  of $N_{kl}$ there exists a unique continuation $\widetilde{h} \in \calO^{*}(\sigma)$.
  Hence we can extend all transition maps $N_{kl}: V_k \cap V_l \rightarrow G(r)$ to maps $N_{kl}':\sigma
  \rightarrow G(r)$. Now we choose for every $i \in I$ a point $P_i \in V_i$ and for all
  pairs $k,l \in I$ a path $\gamma_{kl}:[0,1] \rightarrow \sigma$ from $P_k$ to $P_l$.
  Let $k,l \in I$ be given. As the image of $\gamma_{kl}$ is compact there exists a
  finite covering $V_{\mu_1},\ldots,V_{\mu_c}$ of
  $\gamma_{kl}([0,1])$. For $x \in V_l$ we set $$S(\gamma_{kl})(x):=
  (N_{\mu_{1},\mu_{2}}'(x))^{-1} \odot \cdots \odot (N_{\mu_{c-1},\mu_{c}}'(x))^{-1} \in
  G(r).$$ Now fix some $k_0 \in I$. For all $l \in I$ we define maps
  $$\widetilde{\varphi}_{\pi^{-1}(\sigma)}^{(l)}: V_l \times \RR^r \cong \pi^{-1}(V_l) \rightarrow \RR^{n_\sigma + r}
  : (x,a) \mapsto (\varphi_\sigma(x),S(\gamma_{k_0 l})(x) \odot a).$$
  These maps agree on overlaps and hence glue together to an embedding
  $$\widetilde{\varphi}_{\pi^{-1}(\sigma)}: \pi^{-1}(\sigma) \rightarrow \RR^{n_\sigma + r}.$$
  In the same way we can construct the fan charts $\widetilde{\Phi}_{\pi^{-1}(\sigma)}$.
  Then we define $F$ to be the equivalence class
  $$F:=\left[(((F_0,|F_0|,\{ \widetilde{\varphi}_{\pi^{-1}(\sigma)} \}),\omega_{F_0}), \{\widetilde{\Phi}_{\pi^{-1}(\sigma)} \})
  \right].$$
\end{remark}

\begin{example} \label{ex-ellipticcurve}
  Throughout the chapter, the curve $X:=X_2$ from \cite[example 5.5]{AR07} will serve us as a central
  example. Recall that $X$ arises by gluing open fans as
  drawn in the figure:
  \begin{center}
    \begin{picture}(0,0)%
\includegraphics{pic/Ex_EllipCurve.pstex}%
\end{picture}%
\setlength{\unitlength}{2072sp}%
\begingroup\makeatletter\ifx\SetFigFont\undefined%
\gdef\SetFigFont#1#2#3#4#5{%
  \reset@font\fontsize{#1}{#2pt}%
  \fontfamily{#3}\fontseries{#4}\fontshape{#5}%
  \selectfont}%
\fi\endgroup%
\begin{picture}(2969,3058)(6639,-3615)
\put(6931,-1591){\makebox(0,0)[lb]{\smash{{\SetFigFont{6}{7.2}{\familydefault}{\mddefault}{\updefault}{\color[rgb]{0,0,0}$X$}%
}}}}
\end{picture}%

  \end{center}
  Moreover, recall from \cite[definition 5.4]{AR07} that the transition functions
  between these open fans composing $X$ are integer affine linear. This
  implies that the curve $X$ has a well-defined lattice length
  $L$. We can cover $X$ by open sets $U_1, U_2, U_3$ as drawn in the following figure:
  
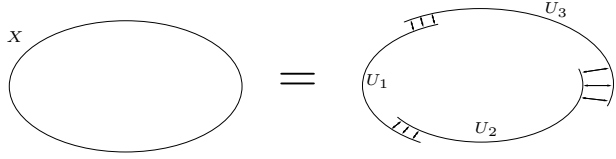
\begin{center}
    \begin{picture}(0,0)%
\includegraphics{pic/elliptic_curve.pstex}%
\end{picture}%
\setlength{\unitlength}{1973sp}%
\begingroup\makeatletter\ifx\SetFigFont\undefined%
\gdef\SetFigFont#1#2#3#4#5{%
  \reset@font\fontsize{#1}{#2pt}%
  \fontfamily{#3}\fontseries{#4}\fontshape{#5}%
  \selectfont}%
\fi\endgroup%
\begin{picture}(7648,2029)(3851,-3665)
\put(7276,-2761){\makebox(0,0)[lb]{\smash{{\SetFigFont{6}{7.2}{\familydefault}{\mddefault}{\updefault}{\Huge \color[rgb]{0,0,0}$=$}%
}}}}
\put(3866,-2136){\makebox(0,0)[lb]{\smash{{\SetFigFont{6}{7.2}{\familydefault}{\mddefault}{\updefault}{\color[rgb]{0,0,0}$X$}%
}}}}
\put(8381,-2676){\makebox(0,0)[lb]{\smash{{\SetFigFont{6}{7.2}{\familydefault}{\mddefault}{\updefault}{\color[rgb]{0,0,0}$U_1$}%
}}}}
\put(10626,-1771){\makebox(0,0)[lb]{\smash{{\SetFigFont{6}{7.2}{\familydefault}{\mddefault}{\updefault}{\color[rgb]{0,0,0}$U_3$}%
}}}}
\put(9751,-3276){\makebox(0,0)[lb]{\smash{{\SetFigFont{6}{7.2}{\familydefault}{\mddefault}{\updefault}{\color[rgb]{0,0,0}$U_2$}%
}}}}
\end{picture}%

  \end{center}
  The easiest way to construct a (non-trivial) vector bundle of rank $r$ on $X$ is
  fixing a (non-trivial) transition map $M_{12}: U_1 \cap U_2 \rightarrow G(r)$ and
  defining $M_{23}: U_2 \cap U_3 \rightarrow G(r)$, $M_{31}:U_3 \cap U_1 \rightarrow G(r)$ to be the trivial maps
  $x \mapsto E$ for all $x$. We will see later that in fact every vector
  bundle of rank $r$ on $X$ arises in this way.
\end{example}

Knowing what tropical vector bundles are, there are a few notions
related to this definition we want to introduce now:

\begin{definition}[Direct sums of vector bundles]
  Let $\pi_1: F_1 \rightarrow X$ and $\pi_2: F_2 \rightarrow X$ be
  two vector bundles of rank $r$ and $r'$, respectively, with a common open covering $U_1,\ldots,U_s$
  and transition maps $M_{ij}^{(1)}$ and $M_{ij}^{(2)}$,
  respectively, satisfying definition \ref{def-vectorbundle} (see remark \ref{remark-refinedcovering}).
  We define the \emph{direct sum bundle} $\pi: F_1 \oplus F_2 \rightarrow X$ to
  be the vector bundle of rank $r+r'$ we obtain from the gluing data
  \begin{itemize}
    \item $U_1,\ldots,U_s$
    \item $M_{ij}^{(1)} \times M_{ij}^{(2)}: U_i \cap U_j \rightarrow
    G(r+r'): x \mapsto \left( \begin{array}{cc} M_{ij}^{(1)}(x) & -\infty \\ -\infty & M_{ij}^{(2)}(x) \end{array} \right)$.
  \end{itemize}
\end{definition}

\begin{definition}[Subbundles] \label{def-subbundles}
  Let $\pi: F \rightarrow X$ be a vector bundle with open covering
  $U_1,\ldots,U_s$ and homeomorphisms $\Phi_i$ according to definition
  \ref{def-vectorbundle}. A subcycle $E \in Z_l(F)$ is called a subbundle of rank $r'$ of $F$ if
  $\pi|_{E}: E \rightarrow X$ is a vector bundle of rank $r'$ such
  that we have for all $i=1,\ldots,s$: $$\Phi_i|_{(\pi|_{E})^{-1}(U_i)}: (\pi|_{E})^{-1}(U_i)
  \stackrel{\cong}{\rightarrow} U_i \times \langle
  e_{j_1},\ldots,e_{j_{r'}}\rangle_{\RR}$$
  for some $1\leq j_1< \ldots < j_{r'} \leq r$, where the $e_j$ are the standard basis vectors in $\RR^r$.
\end{definition}

\begin{remark} \label{remark-subbundles}
  If $\pi: F \rightarrow X$ is a vector bundle of rank $r$ with
  subbundle $E$ of rank $r'$ like in definition \ref{def-subbundles} this implies that there exists another subbundle
  $E'$ of rank $r-r'$ with $$\Phi_i|_{(\pi|_{E'})^{-1}(U_i)}: (\pi|_{E'})^{-1}(U_i)
  \stackrel{\cong}{\rightarrow} U_i \times \langle
  e_j| j \not\in \{j_1,\ldots,j_{r'} \}\rangle_{\RR}$$ and hence
  that $F=E \oplus E'$ holds.
\end{remark}

\begin{definition}[Decomposable bundles]
  Let $\pi: F \rightarrow X$ be a vector bundle of rank $r$. We say that $F$
  is \emph{decomposable} if there exists a subbundle $\pi|_{E}: E \rightarrow
  X$ of $F$ of rank $1 \leq r' < r$. Otherwise we call
  $F$ an \emph{indecomposable vector bundle}.
\end{definition}

As announced in the very beginning of this section we also want to
talk about morphisms and, in particular, isomorphisms of tropical
vector bundles:

\begin{definition}[Morphisms of vector bundles] \label{def-morphism}
  A morphism of vector bundles $\pi_1:F_1 \rightarrow X$ of rank $r$ and $\pi_2:F_2 \rightarrow X$
  of rank $r'$ is a morphism $\Psi: F_1 \rightarrow F_2$ of tropical cycles such that
  \begin{enumerate}
    \item $\pi_1=\pi_2 \circ \Psi$ and
    \item there exist an open covering $U_1,\ldots,U_s$ according to definition \ref{def-vectorbundle}
    for both $F_1$ and $F_2$ (see remark \ref{remark-refinedcovering}) and maps $A_i: U_i \rightarrow G(r' \times r)$ for all $i$ such that
    $$\Phi_i^{F_2} \circ \Psi \circ (\Phi_i^{F_1})^{-1}:U_i \times \RR^r \rightarrow U_i \times \RR^{r'}$$
    is given by $(x,a) \mapsto (x,f_{A_i(x)}(a))$ (cf. \ref{remark-map_f_A}) and the entries of $A_i$ are
    regular invertible functions on $U_i$ or constantly $-\infty$.
  \end{enumerate}
  An isomorphism of tropical vector bundles is a morphism of
  vector bundles $\Psi: F_1 \rightarrow F_2$ such that there exists a
  morphism of vector bundles $\Psi': F_2 \rightarrow F_1$ with
  $\Psi' \circ \Psi = \id = \Psi \circ \Psi'$.
\end{definition}

\begin{lemma} \label{lemma-isomorphicbundles}
  Let $\pi_1:F_1 \rightarrow X$ and $\pi_2:F_2 \rightarrow X$ be
  two vector bundles of rank $r$ over $X$. Then the following are equivalent:
  \begin{enumerate}
    \item There exists an isomorphism of vector bundles $\Psi: F_1 \rightarrow F_2$.
    \item There exist a common open covering $U_1,\ldots,U_s$ of $X$ and transition maps
    $M_{ij}^{(1)}$ for $F_1$ and $M_{ij}^{(2)}$ for $F_2$ satisfying definition \ref{def-vectorbundle}
    (cf. remark \ref{remark-refinedcovering}) and maps ${E_i: U_i \rightarrow G(r)}$ for $i=1,\ldots,s$ such that
      \begin{itemize}
        \item the entries of $E_i$ are regular invertible functions on $U_i$ or constantly $-\infty$ and
        \item for all $i,j$ holds $E_j(x) \odot M_{ij}^{(1)}(x)=M_{ij}^{(2)}(x) \odot E_i(x)$ for all $x \in U_i \cap U_j$.
      \end{itemize}
  \end{enumerate}
\end{lemma}
\begin{proof}
  $(a) \Rightarrow (b)$: We claim that the maps $A_i: U_i \rightarrow G(r \times r)$ of definition
  \ref{def-morphism} are the wanted maps $E_i$. As $\Psi$ is an isomorphism we can conclude
  that $A_i(x)$ is an invertible matrix for all $x \in U_i$, i.e. that $A_i: U_i \rightarrow G(r)$.
  Hence it remains to check
  that $A_j(x) \odot M_{ij}^{(1)}(x)=M_{ij}^{(2)}(x) \odot A_i(x)$ holds for all $x \in U_i \cap
  U_j$: Let $i,j$ be given. As $\Psi: F_1 \rightarrow F_2$ is an
  isomorphism, the diagram
  $$\begin{xy}
    \xymatrix@C=2.5cm{
      (U_i \cap U_j) \times \RR^r \ar[r]^{\Phi_i^{F_2} \circ \Psi \circ (\Phi_i^{F_1})^{-1}} \ar[d]_{\Phi_j^{F_1} \circ (\Phi_i^{F_1})^{-1}} &
      (U_i \cap U_j) \times \RR^r \ar[d]^{\Phi_j^{F_2} \circ (\Phi_i^{F_2})^{-1}}\\
      (U_i \cap U_j) \times \RR^r \ar[r]_{\Phi_j^{F_2} \circ \Psi \circ (\Phi_j^{F_1})^{-1}} & (U_i \cap U_j) \times \RR^r
    }
  \end{xy}$$
  commutes. Hence $A_j(x) \odot M_{ij}^{(1)}(x)=M_{ij}^{(2)}(x) \odot
  A_i(x)$ holds.\\
  $(b) \Rightarrow (a)$: Conversely, let the maps $E_i: U_i
  \rightarrow G(r)$ be given. The equation $$E_j(x) \odot M_{ij}^{(1)}(x)=M_{ij}^{(2)}(x) \odot
  E_i(x)$$ for all $x \in U_i \cap U_j$ ensures that the maps
  $$U_i \times \RR^r \rightarrow U_i \times \RR^r: (x,a) \mapsto (x,E_i(x) \odot
  a)$$ on the local trivializations can be glued to a globally defined
  map $\Psi: |F_1| \rightarrow |F_2|$. Moreover, this map is a morphism
  as $\pi_1,\pi_2$ are morphisms and the maps $p_j^{(i)} \circ \Phi_i^{F_1}$,
  $p_j^{(i)} \circ \Phi_i^{F_2}$ and the finite entries of $E_i$ are regular
  invertible functions (cf. definition \ref{def-vectorbundle}).
  The equation $E_j(x) \odot M_{ij}^{(1)}(x)=M_{ij}^{(2)}(x) \odot
  E_i(x)$ implies that $$E_j^{-1}(x) \odot M_{ij}^{(2)}(x)=M_{ij}^{(1)}(x) \odot
  E_i^{-1}(x)$$ holds for all $x \in U_i \cap U_j$, where $E_k^{-1}(x) := (E_k(x))^{-1}$
  for all $x \in U_k$. As the finite entries of $E_k^{-1}: U_k \rightarrow G(r)$
  are again regular invertible functions we can also glue the maps
  $$U_i \times \RR^r \rightarrow U_i \times \RR^r: (x,a) \mapsto (x,E_i^{-1}(x) \odot a)$$
  on the local trivializations to obtain the inverse morphism $\Psi': |F_2| \rightarrow
  |F_1|$, which proves that $\Psi$ is an isomorphism.
\end{proof}

The morphisms we have just introduced admit another important
operation, namely the pull-back of a vector bundle:

\begin{definition}[Pull-back of vector bundles]
  Let $\pi:F \rightarrow X$ be a vector bundle of rank $r$ with open covering
  $U_1,\ldots,U_s$ and transition maps $M_{ij}$ as in definition \ref{def-vectorbundle}, and let $f:Y
  \rightarrow X$ be a morphism of tropical cycles. Then the
  \emph{pull-back bundle} $\pi':f^{*}F \rightarrow Y$ is the vector
  bundle we obtain by gluing the patches $f^{-1}(U_1) \times \RR^r,\ldots,f^{-1}(U_s) \times \RR^r$
  along the transition maps $M_{ij} \circ f$. Hence we obtain the
  commutative diagram
  $$\begin{xy}
    \xymatrix{
      f^{*}F \ar@{-->}[r]^{f'} \ar@{-->}[d]_{\pi'} & F \ar[d]^{\pi}\\
      Y \ar[r]_f & X
    }
  \end{xy}$$
  where $f'$ and $\pi'$ are locally given by $f':f^{-1}(U_i) \times \RR^r \rightarrow U_i \times \RR^r:
  (y,a) \mapsto (f(y),a)$ and $\pi':f^{-1}(U_i) \times \RR^r \rightarrow f^{-1}(U_i): (y,a) \mapsto y$.
\end{definition}

To be able to define Chern classes in the second section we need
the notion of a rational section of a vector bundle:

\begin{definition}[Rational sections of vector bundles] \label{def-sectionsofbundles}
  Let $\pi:F \rightarrow X$ be a vector bundle of rank $r$.
  A \emph{rational section} $s:X \rightarrow F$ of $F$ is a continuous map $s:|X| \rightarrow |F|$ such that
  \begin{enumerate}
    \item $\pi(s(x))=x$ for all $x \in |X|$ and
    \item there exist an open covering $U_1,\ldots,U_s$ and homeomorphisms $\Phi_i$
    satisfying definition \ref{def-vectorbundle} (cf. definition \ref{def-identifyequalbundles})
    such that the maps $p_j^{(i)} \circ \Phi_i \circ s: U_i \rightarrow
    \RR$ are rational functions on $U_i$ for all $i,j$,
  \end{enumerate}
  where $p_j^{(i)}: U_i \times \RR^r \rightarrow \RR$ is given by $(x,(a_1,\ldots,a_r)) \mapsto
  a_j$. A rational section $s: X \rightarrow F$ is called \emph{bounded} if the
  above maps $p_j^{(i)} \circ \Phi_i \circ s$ are bounded for all
  $i,j$.
\end{definition}

\begin{remark}
  Let $\pi: L \rightarrow X$ be a line bundle and $s: X
  \rightarrow L$ a rational section. By definition, the map $p^{(i)} \circ \Phi_i
  \circ s$ is a rational function on $U_i$ for all $i$. Moreover,
  on $U_i \cap U_j$ the maps $p^{(i)} \circ \Phi_i \circ s$ and
  $p^{(j)} \circ \Phi_j \circ s$ differ
  by a regular invertible function only. Hence $s$ defines a
  Cartier divisor $\calD(s) \in \Div(X)$.
\end{remark}

There is a useful statement on these Cartier divisors $\calD(s)$
in \cite{T09} that we want to cite here including its proof:

\begin{lemma} \label{lemma-boundedsectionsinlinebundleareequivalent}
  Let $\pi: L \rightarrow X$ be a line bundle and let $s_1, s_2: X
  \rightarrow L$ be two bounded rational sections. Then $\calD(s_1)-\calD(s_2) = h$ for
  some bounded rational function $h \in \mathcal{K}^{*}(X)$, i.e.
  $\calD(s_1)$ and $\calD(s_2)$ are rationally equivalent.
\end{lemma}
\begin{proof}
  Let $U_1,\ldots,U_s$ be an open covering of $X$ with transition maps
  $M_{ij}$ and homeomorphisms $\Phi_i$ according to definition \ref{def-vectorbundle}
  such that for all $i$ both $s_1^{(i)}:=p^{(i)} \circ \Phi_i \circ s_1$ and $s_2^{(i)}:=p^{(i)} \circ \Phi_i \circ s_2$
  are rational functions on $U_i$ (cf. definition \ref{def-sectionsofbundles}).
  We define $h_i := s_1^{(i)}-s_2^{(i)} \in \calK^{*}(U_i).$
  As we have $s_1^{(i)}-s_1^{(j)} = s_2^{(i)}-s_2^{(j)} = M_{ij} \in \calO^{*}(U_i \cap
  U_j)$ for all $i,j$ these maps $h_i$ glue together to $h \in \calK^{*}(X)$. Hence we have
  $$\begin{array}{rcl}
    \calD(s_1)-\calD(s_2) &=& [\{(U_i,s_1^{(i)})\}]-[\{(U_i,s_2^{(i)})\}] \\
    &=& [\{(U_i,s_1^{(i)}-s_2^{(i)})\}]\\
    &=& [\{(U_i,h_i)\}]\\
    &=& [\{(|X|,h)\}].
  \end{array}$$
\end{proof}

\begin{remark} \label{remark-welldefinedcartierdivisorclass}
  Lemma \ref{lemma-boundedsectionsinlinebundleareequivalent}
  implies that we can associate to any line bundle $L$ admitting a bounded
  rational section $s$ a Cartier divisor class $\calD(F):=[\calD(s)]$ that only
  depends on the bundle $L$ and not on the choice of the rational section $s$.
\end{remark}

Combining both the notion of a morphism of vector bundles and the
notion of a rational section we can define the following:

\begin{definition}[Pull-back of rational sections]
  Let $\pi: F \rightarrow X$ be a vector bundle of rank $r$ and $f:Y
  \rightarrow X$ a morphism of tropical varieties. Moreover, let $s:X \rightarrow F$ be a rational section
  of $F$ with open covering $U_1,\ldots,U_s$ and homeomorphisms $\Phi_1,\ldots,\Phi_s$ as in definition
  \ref{def-sectionsofbundles}. Then we can define a rational section
  $f^{*}s:~Y~\rightarrow~f^{*}F$ of $f^{*}F$, the \emph{pull-back section} of $s$, as follows: On
  $f^{-1}(U_i)$ we define $$f^{*}s: f^{-1}(U_i) \rightarrow
  f^{-1}(U_i) \times \RR^r: y \mapsto (y,(p_i \circ \Phi_i \circ s \circ
  f)(y)),$$
  where $p_i:U_i \times \RR^r \rightarrow \RR^r$ is the projection
  on the second factor. Note that for $y \in f^{-1}(U_i) \cap
  f^{-1}(U_j)$ the points $(y,(p_i \circ \Phi_i \circ s \circ
  f)(y))$ and $(y,(p_j \circ \Phi_j \circ s \circ f)(y))$ are
  identified in $f^{*}F$ if and only if $(f(y),(p_i \circ \Phi_i \circ s \circ
  f)(y))$ and $(f(y),(p_j \circ \Phi_j \circ s \circ f)(y))$ are
  identified in $F$. But this is the case as $(f(y),(p_i \circ \Phi_i \circ s \circ
  f)(y)) = (\Phi_i \circ s)(f(y)) \sim (\Phi_j \circ s)(f(y)) =
  (f(y),(p_j \circ \Phi_j \circ s \circ f)(y)).$ Hence we can glue our locally
  defined map $f^{*}s$ to obtain a map $f^{*}s: Y \rightarrow
  f^{*}F$.
\end{definition}

We finish this section with the following statement on vector
bundles on simply connected tropical cycles which will be of use
for us later on:

\begin{theorem} \label{thm-bundleonsimplyconnectedspacesplits}
  Let $\pi:F \rightarrow X$ be a vector bundle of rank $r$ on the
  simply connected tropical cycle $X$. Then $F$ is a direct sum
  of line bundles, i.e. there exist line bundles $L_1,\ldots,L_r$ on
  $X$ such that $F=L_1 \oplus \ldots \oplus  L_r$.
\end{theorem}
\begin{proof}
  We show that every vector bundle of rank $r \geq 2$
  on $X$ is decomposable. Let $U_1, \ldots, U_s$ be an open covering of
  $X$ and let $$M_{ij}(x)=D(\varphi_{i,j}^{(1)},\ldots,\varphi_{i,j}^{(r)})(x) \odot
  A_{\sigma_{ij}}(x) =: D_{ij}(x) \odot A_{\sigma_{ij}}(x), \enspace x \in U_i \cap U_j$$
  with $\varphi_{i,j}^{(1)},\ldots,\varphi_{i,j}^{(r)} \in \mathcal{O}^{*}(U_i \cap U_j)$
  and $\sigma_{ij}(x) \in S_r$ be transition functions according to definition \ref{def-vectorbundle}.
  We only have to show that it is possible to track the first coordinate of the $\RR^r$-factor in $U_1
  \times \RR^r$ consistently along the transition maps: Let
  $\gamma: [0,1] \rightarrow |X|$ be a closed path starting and
  ending in $P \in U_1$. Decomposing $\gamma$ into several paths if necessary, we
  may assume that $\gamma$ has no self-intersections, i.e. that $\gamma|_{[0,1)}$ is injective.
  As $\gamma([0,1])$ is compact we can
  choose an open covering $V_1,\ldots,V_t$ of $\gamma([0,1])$ such
  that for all $j$ we have $V_j \subseteq U_i$ for some index $i=i(j)$, $P \in V_1=V_t \subseteq U_1$,
  all sets $V_j$ and all intersections $V_j \cap V_{j+1}$ are connected and
  all intersections $V_j \cap V_{j'}$ for non-consecutive indices are
  empty. For sets $V_j$ and $V_{j'}$ with non-empty intersection we have restricted
  transition maps $\widetilde{M}_{V_j,V_{j'}}(x)=\widetilde{D}_{V_j,V_{j'}}(x) \odot
  A_{\sigma_{V_j,V_{j'}}}$ induced by the transition maps between $U_{i(j)} \supseteq V_j$
  and $U_{i(j')} \supseteq V_{j'}$. Note that the permutation parts $A_{\sigma_{V_j,V_{j'}}}$ of
  the transition maps do not depend on $x$ as all intersections $V_j \cap V_{j'}$ are connected
  and the permutations have to be locally constant. We define
  $I_\gamma := \sigma_{V_{t-1},V_t} \circ \ldots \circ \sigma_{V_1,V_2}(1)$.
  We have to check that $I_\gamma=1$ holds.
  First we show that $I_\gamma$ does not depend on the choice of
  the covering $V_1,\ldots,V_t$. Hence, let
  $V_1',\ldots,V_{t'}'$ be another covering as above. We may assume that all
  intersections $V_j \cap V_{j'}'$ are connected, too. Between any two sets $A,B \in \{V_1,\ldots,V_t,V_{1}',\ldots,V_{t'}'\}$
  with non-empty intersection we have restricted transition maps $\widetilde{M}_{A,B}(x)=
  \widetilde{D}_{A,B}(x) \odot A_{\sigma_{A,B}}$ as above. Moreover, let
  $0=\alpha_0<\alpha_1<\ldots<\alpha_p=1$ be a decomposition of
  $[0,1]$ such that for all $i$ we have $\gamma([\alpha_i,\alpha_{i+1}]) \subseteq V_j
  \cap V_{j'}'$ for some indices $j,j'$. Let $i_0$ be the maximal
  index such that $\gamma([\alpha_{i_0},\alpha_{{i_0}+1}])
  \subseteq V_a \cap V_b'$ and $$\sigma_{V_{a-1},V_a} \circ \ldots \circ \sigma_{V_1,V_2}
  = \sigma_{V_b',V_a} \circ \sigma_{V_{b-1}',V_b'} \circ \ldots \circ
  \sigma_{V_1',V_2'}$$
  is still fulfilled. Assume that $i_0<p-1$. Let $\gamma([\alpha_{{i_0}+1},\alpha_{{i_0}+2}])
  \subseteq V_{a'} \cap V_{b'}'$. Hence $\gamma(\alpha_{{i_0}+1})
  \in V_{a} \cap V_{b}' \cap V_{a'} \cap V_{b'}'$ and we can
  conclude using the cocycle condition:
  $$\begin{array}{rcl}
    \sigma_{V_{a},V_{a'}} \circ \sigma_{V_{a-1},V_a} \circ \ldots \circ \sigma_{V_1,V_2} &=&
    \sigma_{V_{a},V_{a'}} \circ \sigma_{V_b',V_a} \circ \sigma_{V_{b-1}',V_b'} \circ \ldots \circ \sigma_{V_1',V_2'}\\
    &=& \sigma_{V_{a},V_{a'}} \circ \sigma_{V_{b'}',V_{a}} \circ \sigma_{V_b',V_{b'}'} \circ \sigma_{V_{b-1}',V_b'} \circ \ldots \circ \sigma_{V_1',V_2'}\\
    &=& \sigma_{V_{b'}',V_{a'}} \circ \sigma_{V_b',V_{b'}'} \circ \sigma_{V_{b-1}',V_b'} \circ \ldots \circ \sigma_{V_1',V_2'},
  \end{array}$$
  a contradiction to our assumption. Hence $i_0=p-1$ and we can conclude that $I_\gamma$ is independent of the chosen
  covering.\\
  If $\gamma$ and $\gamma'$ are paths that pass through
  exactly the same open sets $U_i$ in the same order, then we can conclude
  that $I_\gamma=I_{\gamma'}$ holds as exactly the same transition
  functions are involved. Hence, a continuous deformation of $\gamma$ does not change
  $I_\gamma$. As $X$ is simply connected we can contract $\gamma$ to
  a point. This implies $I_\gamma = I_{\gamma_0}$, where $\gamma_0$ is
  the constant path $\gamma_0(t)=P$ for all $t$. Thus
  $I_\gamma=I_{\gamma_0}=1$. This proves the claim.
\end{proof}

There is a related theorem in \cite{T09} which we want to state
here. As we will not need the result in this work, we will omit the
proof and refer to \cite{T09} instead.

\begin{theorem} \label{thm-bundleonsimplyconnectedspacetrivial}
  Let $\pi:L \rightarrow X$ be a line bundle on the
  simply connected tropical cycle $X$. Then $L$ is trivial, i.e.
  $L \cong X \times \RR$ as a vector bundle.
\end{theorem}

Combing both theorem \ref{thm-bundleonsimplyconnectedspacesplits}
and theorem \ref{thm-bundleonsimplyconnectedspacetrivial} we can
conclude the following:

\begin{corollary}
  Let $\pi:F \rightarrow X$ be a vector bundle of rank $r$ on the
  simply connected tropical cycle $X$. Then $F$ is trivial, i.e.
  $F \cong X \times \RR^r$ as a vector bundle.
\end{corollary}

\section{Chern classes} \label{sec-chernclasses}

In this section we will introduce Chern classes of tropical vector
bundles and prove basic properties. To be able to do this we need
some preparation:

\begin{definition} \label{def-intersectwithsection}
  Let $\pi: F \rightarrow X$ be a vector bundle of rank $r$ and let $s: X \rightarrow F$ be a rational section
  with open covering $U_1,\ldots,U_s$ as in definition \ref{def-sectionsofbundles}. We fix a natural
  number $1 \leq k \leq r$ and a subcycle $Y \in Z_l(X)$.
  By definition, $s_{ij}:=p_j^{(i)} \circ \Phi_i \circ s: U_i \rightarrow
  \RR$ is a rational function on $U_i$ for all $i,j$. Hence, for all $i$ we can
  take local intersection products
  $$(s^{(k)} \cdot Y) \cap U_i  := \sum_{1\leq j_1<\ldots<j_k \leq r} s_{ij_1} \cdots s_{ij_k} \cdot (Y \cap U_i).$$
  Since $s_{i'j} = s_{i \sigma(j)}+\varphi_j$ on $U_i \cap U_{i'}$ for some $\sigma \in S_r$ and some regular
  invertible map $\varphi_j \in \mathcal{O}^{*}(U_i \cap U_{i'})$,
  the intersection products $(s^{(k)} \cdot Y) \cap U_i$ and $(s^{(k)} \cdot Y) \cap U_{i'}$ coincide on $U_i
  \cap U_{i'}$ and we can glue them to obtain a global intersection
  cycle $s^{(k)} \cdot Y \in Z_{l-k}(X)$.
\end{definition}

\begin{lemma} \label{lemma-rationalequivalence}
  Let $\pi: F \rightarrow X$ be a vector bundle of rank $r$, fix $k \in \{1,\ldots,r\}$
  and let $s: X \rightarrow F$ be a rational section. Moreover, let $Y \in
  Z_l(X)$ be a cycle and let $\varphi \in \mathcal{K}^{*}(Y)$ be a
  bounded rational function on $Y$. Then the following equation holds:
  $$s^{(k)} \cdot (\varphi \cdot Y) = \varphi \cdot (s^{(k)} \cdot Y).$$
\end{lemma}
\begin{proof}
  The claim follows immediately from the definition of the product
  $s^{(k)} \cdot Y$.
\end{proof}

\begin{lemma} \label{lemma-equivalentsectionsinsiobundles}
  Let $\pi: F \rightarrow X$ and $\pi': F' \rightarrow X$ be two
  isomorphic vector bundles of rank $r$ with isomorphism $f:F \rightarrow F'$.
  Moreover, fix $k \in \{1,\ldots,r\}$, let $s: X \rightarrow F$ be a rational section and let $Y \in
  Z_l(X)$ be a cycle. Then the following equation holds: $$s^{(k)} \cdot Y = (f \circ s)^{(k)} \cdot
  Y \in Z_{l-k}(X).$$
\end{lemma}
\begin{proof}
  Let $U_1,\ldots,U_s$ be an open covering of $X$ satisfying definition
  \ref{def-vectorbundle} for both $F$ and $F'$ and let $s_{ij}:=p_j^{(i)} \circ \Phi_i \circ s: U_i \rightarrow
  \RR$ and $(f \circ s)_{ij}:=p_j^{(i)} \circ \Phi_i \circ f \circ s: U_i \rightarrow
  \RR$ as in definition \ref{def-intersectwithsection}. By lemma \ref{lemma-isomorphicbundles} the
  isomorphism $f$ can be described on $U_i \times \RR^r$ by $(x,a)
  \mapsto (x,E_i(x) \odot a)$ with $E_i(x)=D(\varphi_1,\ldots,\varphi_r) \odot A_\sigma$ for some
  regular invertible functions $\varphi_1,\ldots,\varphi_r \in \mathcal{O}^{*}(U_i)$ and a permutation
  $\sigma \in S_r$. Hence $(f \circ s)_{ij} =
  s_{i\sigma(j)}+\varphi_j$ on $U_i$ and thus $$\sum_{1\leq j_1<\ldots<j_k \leq r} s_{ij_1} \cdots s_{ij_k} \cdot (Y \cap U_i)
  = \sum_{1\leq j_1<\ldots<j_k \leq r} (f \circ s)_{ij_1} \cdots (f \circ s)_{ij_k} \cdot (Y \cap
  U_i),$$ which proves the claim.
\end{proof}

To be able to prove the next theorem, which will be essential for
defining Chern classes, we first need some generalizations of our
previous definitions:

\begin{definition}[Infinite tropical cycle] \label{def-infinitecycle}
  We define an \emph{infinite tropical polyhedral complex} to be a
  tropical polyhedral complex according to definition
  \cite[definition 5.4]{AR07} but we do not require the set of
  polyhedra $X$ to be finite. In particular, all open fans
  $F_\sigma$ have still to be open tropical fans according to
  \cite[definition 5.3]{AR07}. Then an \emph{infinite tropical
  cycle} is an infinite tropical polyhedral complex modulo
  refinements analogous to \cite[definition 5.12]{AR07}.
\end{definition}

\begin{definition}[Infinite rational functions and infinite Cartier divisors] \label{def-infiniteCartierdivisors}
  Let $C$ be an infinite \linebreak tropical cycle and let $U$ be an open set
  in $|C|$. As in \cite[definition 6.1]{AR07}
  an \emph{infinite rational} \linebreak \emph{function} on $U$ is a continuous function
  $\varphi : U \rightarrow \RR$ such that there exists a representative \linebreak
  $(((X, |X|, \{m_\sigma\}_{\sigma \in X}),\omega_X), \{M_\sigma\}_{\sigma \in
  X})$ of $C$, which may now be an infinite tropical polyhedral complex,
  such that for each face $\sigma \in X$ the map $\varphi \circ m_\sigma^{-1}$ is
  locally integer affine linear (where defined). Analogously it is possible to
  define \emph{infinite regular invertible functions} on $U$.\\
  A \emph{representative of an infinite Cartier divisor} on $C$ is then a
  set $\{(U_i,\varphi_i)| \; i \in I\}$, where $\{U_i\}$ is an open covering of $|C|$
  and $\varphi_i$ is an infinite rational function on $U_i$. An \emph{infinite Cartier
  divisor} on $C$ is then a representative of an infinite Cartier divisor
  modulo the equivalence relation given in \cite[definition 6.1]{AR07}.
\end{definition}

\begin{remark}
  Using these basic definitions it is possible to generalize many
  other concepts to the infinite case. In particular, as our
  infinite objects are locally finite, it is possible to perform
  intersection theory as before.
\end{remark}

\begin{definition}[Tropical vector bundles on infinite cycles] \label{def-infinitevectorbundles}
  Let $X$ be an infinite tropical cycle. A \emph{tropical vector bundle} over
  $X$ of rank $r$ is an infinite tropical cycle $F$ together with a morphism $\pi: F
  \rightarrow X$ such that properties (a)--(d) given in definition
  \ref{def-vectorbundle} are fulfilled with the difference that the open covering
  $\{U_i\}$ of $X$ may now be infinite.
\end{definition}

Now we are ready to prove the announced theorem:

\begin{theorem} \label{thm-chernclasswelldefined}
  Let $\pi: F \rightarrow X$ be a vector bundle of rank $r$ and
  $s_1, s_2: X \rightarrow F$ two bounded rational sections. Then
  $s_1^{(k)} \cdot Y$ and $s_2^{(k)} \cdot Y$ are rationally equivalent, i.e.
  $$[s_1^{(k)} \cdot Y] = [s_2^{(k)} \cdot Y] \in A_{*}(X)$$ holds for all subcycles $Y \in Z_l(X)$.
\end{theorem}
\begin{proof}
  Let $p: |\widetilde{X}| \rightarrow |X|$ be the universal covering
  space of $|X|$. We can locally equip $|\widetilde{X}|$ with the tropical
  structure inherited form $X$ and obtain an infinite tropical cycle
  $\widetilde{X}$ according to definition \ref{def-infinitecycle}. Moreover,
  pulling back $F$ along $p$, we obtain a tropical vector bundle $p^{*}F$ on
  $\widetilde{X}$ according to definition \ref{def-infinitevectorbundles}.
  As $\widetilde{X}$ is simply connected we can conclude by lemma
  \ref{thm-bundleonsimplyconnectedspacesplits} that $p^{*}F=L_1 \oplus \ldots \oplus  L_r$
  for some infinite tropical line bundles $L_1,\ldots,L_r$ on $\widetilde{X}$.
  Hence, the bounded rational sections $p^{*}s_1$ and $p^{*}s_2$ correspond to $r$ infinite tropical Cartier
  divisors as in definition \ref{def-infiniteCartierdivisors} each, which we will denote
  by $\varphi_1,\ldots,\varphi_r$ and $\psi_1,\ldots,\psi_r$, respectively. By lemma
  \ref{lemma-boundedsectionsinlinebundleareequivalent} we can conclude that for all $i$
  these Cartier divisors differ by bounded infinite rational functions only,
  i.e. $\varphi_i-\psi_i=h_i$ for some bounded infinite rational function $h_i$ on $\widetilde{X}$.
  In particular, $$\left(\sum_{1\leq j_1<\ldots<j_k \leq r} \varphi_{j_1} \cdots
  \varphi_{j_k}-\sum_{1\leq j_1<\ldots<j_k \leq r} \psi_{j_1} \cdots
  \psi_{j_k}\right) \cdot \widetilde{X} = \widetilde{h} \cdot \widetilde{\xi_2}
  \cdots \widetilde{\xi_k} \cdot \widetilde{X}$$ with a bounded infinite rational
  function $\widetilde{h}$ and infinite Cartier divisors $\widetilde{\xi_i}$.
  Then we can define a rational function $h$, which is then also bounded, and Cartier divisors $\xi_i$ on $X$ as follows:
  Let $U \subseteq |X|$ and $\widetilde{U} \subseteq |\widetilde{X}|$ be open subsets
  such that $p|_{\widetilde{U}}:~\widetilde{U}~\rightarrow~U$ is
  bijective with inverse map $p':U \rightarrow \widetilde{U}$.
  Then we locally define $h|_U:={(p')}^{*}\widetilde{h}|_{\widetilde{U}}$ and $\xi_i|_U:={(p')}^{*}\widetilde{\xi_i}|_{\widetilde{U}}$. Note that
  $h$ and $\xi_i$ are well-defined as the Cartier divisors $\varphi_i$ and $\psi_i$, respectively, are the
  same on every possible set $\widetilde{U} \stackrel{\cong}{\rightarrow} U$. As we locally have
  $$(s_1^{(k)} \cdot Y) \cap U = p_{*} \left( \sum\limits_{1\leq j_1<\ldots<j_k \leq r} \varphi_{j_1} \cdots \varphi_{j_k} \cdot (p')_{*}(Y \cap U)\right)$$
  and
  $$(s_2^{(k)} \cdot Y) \cap U = p_{*} \left( \sum\limits_{1\leq j_1<\ldots<j_k \leq r} \psi_{j_1} \cdots \psi_{j_k} \cdot (p')_{*}(Y \cap U)\right)$$
  we can conclude that
  $$(s_1^{(k)}-s_2^{(k)}) \cdot Y = h \cdot \xi_2 \cdots \xi_k \cdot Y,$$
  which proves the claim.
\end{proof}

Now we are ready to give a definition of Chern classes:

\begin{definition}[Chern classes] \label{def-chernclasses}
  Let $\pi: F \rightarrow X$ be a vector bundle of rank $r$ admitting bounded rational sections.
  For $k \in \{1,\ldots,r\}$ we define the $k$-th Chern class
  of $F$ to be the endomorphism
  $$c_k(F): A_{*}(X) \rightarrow A_{*}(X): [Y] \mapsto [s^{(k)} \cdot Y],$$
  where $A_{*}(X)= \bigoplus_{i} A_i(X)$ and $s:X \rightarrow F$ is any bounded rational section. Note that
  the map $c_k(F)$ is well-defined by lemma \ref{lemma-rationalequivalence} and independent of the
  choice of the rational section $s$ by theorem \ref{thm-chernclasswelldefined}.
  Moreover, we define $c_0(F):A_{*}(X) \rightarrow A_{*}(X)$ to be
  the identity map and $c_k(F):A_{*}(X) \rightarrow A_{*}(X)$ to
  be the zero map for all $k \not\in \{0,\ldots,r\}$. To stress
  the character of an intersection product of $c_k(F)$ we usually
  write $c_k(F) \cdot Y$ instead of $c_k(F)(Y)$ for $Y \in A_{*}(X)$.
\end{definition}

\begin{remark}
  Note that lemma \ref{lemma-equivalentsectionsinsiobundles}
  implies that isomorphic vector bundles have the same Chern classes.
\end{remark}

As announced in the beginning we finish this section with proving
some basic properties of Chern classes:

\begin{theorem}[Properties of Chern classes] \label{thm-propertiesofchernclasses}
  Let $\pi: F \rightarrow X$ and $\pi': F' \rightarrow X$ be vector bundles of rank $r$ and $r'$, respectively, admitting bounded
  rational sections. Moreover, let $f: \widetilde{X} \rightarrow X$ be a morphism of
  tropical cycles. Then the following holds:
  \begin{enumerate}
    \item $c_i(F)=0$ for all $i \not\in \{0,\ldots,\rank(F)\}$,
    \item $c_i(F) \cdot ( c_j(F') \cdot Y) = c_j(F') \cdot (c_i(F) \cdot Y)$ for all $Y \in A_{*}(X)$,
    \item $f_{*}(c_i(f^{*}F) \cdot Y) = c_i(F) \cdot f_{*}(Y)$ for all $Y \in A_{*}(\widetilde{X})$,
    \item $c_i(f^{*}F) \cdot f^{*}(Y) = f^{*}(c_i(F) \cdot Y)$ for all $Y \in A_{*}(X)$ if $X$ and $\widetilde{X}$ are smooth varieties,
    \item $c_k(F \oplus F') = \sum_{i+j=k} c_i(F) \cdot c_j(F')$
    \item $c_1(F) \cdot Y = \calD(F) \cdot Y$ for all $Y \in A_{*}(X)$ if $r=\rank(F)=1$, where $\calD(F)$ is the Cartier divisor class associated to $F$.
  \end{enumerate}
\end{theorem}
\begin{proof}
  Properties (a) and (e) follow immediately from definition
  \ref{def-chernclasses}, property (b) follows from the fact that
  the intersection product is commutative and property (f) follows
  from remark \ref{remark-welldefinedcartierdivisorclass}.\\
  (c): The projection formula implies
    $$f_{*}(c_i(f^{*}F) \cdot Y)= f_{*}([(f^{*}s)^{(i)} \cdot Y])= [s^{(i)} \cdot f_{*}Y] = c_i(F) \cdot f_{*}Y,$$
    where $s$ is any bounded rational section of $F$.\\
  (d): Applying \cite[theorem 3.2 (c) and (f)]{A09} we obtain
    $$c_i(f^{*}F) \cdot f^{*}Y=[(f^{*}s)^{(i)} \cdot f^{*}Y]=[f^{*}(s^{(i)} \cdot Y)]=f^{*}[s^{(i)} \cdot Y]= f^{*}(c_i(F) \cdot Y),$$
    where $s$ is again any bounded rational section of $F$.
\end{proof}

\begin{remark}
  In ``classical'' algebraic geometry even the following, generalized version of property (e) is true:
  Let $0 \rightarrow F' \rightarrow F \rightarrow F'' \rightarrow 0$
  be an exact sequence of vector bundles, then $c_k(F) = \sum_{i+j=k} c_i(F') \cdot
  c_j(F'')$. In the tropical world it is not entirely clear what an
  exact sequence of tropical vector bundles should be.
  Nevertheless, in some sense the ``classical'' statement is true in tropical
  geometry as well: Let $\pi_1: F_1 \rightarrow X$ and $\pi_2: F_2 \rightarrow
  X$ be tropical vector bundles of rank $r_1$ and $r_2$, respectively,
  and let $U_1,\ldots,U_s$ be an open covering of $X$ such that all requirements of definition
  \ref{def-vectorbundle} are fulfilled for $F_1$ and $F_2$ simultaneously. Moreover,
  let $f:F_1 \rightarrow F_2$ be an injective morphism of tropical vector bundles such that
  $(\Phi_i^{F_2} \circ f \circ (\Phi_i^{F_1})^{-1})(U_i \times
  \RR^{r_1}) = U_i \times \langle e_{i_1},\ldots, e_{i_{r_1}}
  \rangle_{\RR}$ for all $i$, i.e. such that the image of $F_1$
  under $f$ is a subbundle $F'$ of $F_2$ (cf. definition
  \ref{def-subbundles}). Then we can conclude by remark
  \ref{remark-subbundles} that $F_2$ is decomposable into $F_2=F'
  \oplus F'' \cong F_1 \oplus F''$ for some other subbundle $F''$
  of $F_2$. Hence we can conclude by theorem \ref{thm-propertiesofchernclasses}
  that $c_k(F_2) = \sum_{i+j=k} c_i(F_1) \cdot c_j(F'')$.
\end{remark}

\section{Vector bundles on an elliptic curve} \label{sec-ellipticcurve}

In this section we will give a complete classification of all
vector bundles on an elliptic curve up to isomorphism. One
characteristic to distinguish different bundles will be the
following:

\begin{definition}[Degree of a vector bundle]
  Let $X:=X_2$ be the curve from \cite[example 5.5]{AR07} and let
  $\pi: F \rightarrow X$ be a vector bundle of rank $r$. We define the
  \emph{degree} of $F$ to be the number $$\deg(F) := \deg(c_1(F) \cdot
  X).$$
\end{definition}

As already advertised in example \ref{ex-ellipticcurve} vector
bundles on the elliptic curve $X$ can be described by a single
transition function. We will prove this fact in the following
lemma:

\begin{lemma} \label{lemma-needonlyonetransitionmap}
  Again, let $X:=X_2$ be the curve from \cite[example 5.5]{AR07} and let
  $\pi: F \rightarrow X$ be a vector bundle of rank $r$.
  Then $F$ is isomorphic to a vector bundle $\pi': F' \rightarrow X$ that
  admits an open covering $U_1',\ldots,U_s'$ and transition maps
  $M_{ij}'$ such that at most one transition map is non-trivial.
\end{lemma}
\begin{proof}
  Let $U_1,\ldots,U_s$ be the open covering with transition maps
  $M_{ij}$ for $F$ according to definition \ref{def-vectorbundle}.
  We may assume that all sets $U_i$ are connected and that for all $i,j$ the intersections $U_i \cap
  U_j$ are connected as well. Moreover, we may assume that the sets $U_i$ are numbered
  consecutively as shown in the figure. For simplicity of notation we will consider our indices modulo $s$.
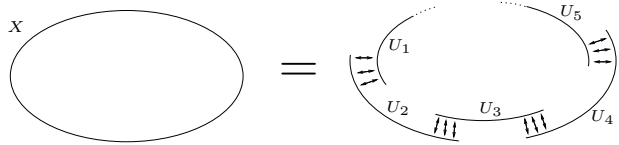
  \begin{center}
    \begin{picture}(0,0)%
\includegraphics{pic/elliptic_curve2.pstex}%
\end{picture}%
\setlength{\unitlength}{1973sp}%
\begingroup\makeatletter\ifx\SetFigFont\undefined%
\gdef\SetFigFont#1#2#3#4#5{%
  \reset@font\fontsize{#1}{#2pt}%
  \fontfamily{#3}\fontseries{#4}\fontshape{#5}%
  \selectfont}%
\fi\endgroup%
\begin{picture}(7667,2169)(3851,-3583)
\put(3866,-2136){\makebox(0,0)[lb]{\smash{{\SetFigFont{6}{7.2}{\familydefault}{\mddefault}{\updefault}{\color[rgb]{0,0,0}$X$}%
}}}}
\put(7276,-2761){\makebox(0,0)[lb]{\smash{{\SetFigFont{6}{7.2}{\familydefault}{\mddefault}{\updefault}{\color[rgb]{0,0,0}\Huge $=$}%
}}}}
\put(9791,-3146){\makebox(0,0)[lb]{\smash{{\SetFigFont{6}{7.2}{\familydefault}{\mddefault}{\updefault}{\color[rgb]{0,0,0}$U_3$}%
}}}}
\put(11196,-3246){\makebox(0,0)[lb]{\smash{{\SetFigFont{6}{7.2}{\familydefault}{\mddefault}{\updefault}{\color[rgb]{0,0,0}$U_4$}%
}}}}
\put(8646,-2351){\makebox(0,0)[lb]{\smash{{\SetFigFont{6}{7.2}{\familydefault}{\mddefault}{\updefault}{\color[rgb]{0,0,0}$U_1$}%
}}}}
\put(8626,-3136){\makebox(0,0)[lb]{\smash{{\SetFigFont{6}{7.2}{\familydefault}{\mddefault}{\updefault}{\color[rgb]{0,0,0}$U_2$}%
}}}}
\put(10801,-1936){\makebox(0,0)[lb]{\smash{{\SetFigFont{6}{7.2}{\familydefault}{\mddefault}{\updefault}{\color[rgb]{0,0,0}$U_5$}%
}}}}
\end{picture}%

  \end{center}
  We can write every map $M_{i,i+1}$, $i=1,\ldots,s$, as
  $$M_{i,i+1}(x)=D(\varphi_{i,i+1}^{(1)},\ldots,\varphi_{i,i+1}^{(r)})(x) \odot
  A_{\sigma_{i,i+1}} =: D_{i}(x) \odot P_{i}$$ for some regular invertible functions
  $\varphi_{i,i+1}^{(k)} \in \mathcal{O}^{*}(U_i \cap U_{i+1})$ and
  permutations $\sigma_{i,i+1} \in S_r$.
  We will show that we can replace successively
  all the transition maps $M_{i,i+1}$ but one by the constant map $M_{i,i+1}': U_i \cap U_{i+1} \rightarrow G(r): x \mapsto E$
  and the resulting vector bundle $F'$ is isomorphic to $F$:
  Choose $j_0 \in \{2,\ldots,s\}$. Note that if we are given a regular invertible
  function $\varphi \in \mathcal{O}^{*}(U_i \cap U_{j})$ there is a unique
  regular invertible function $\widetilde{\varphi} \in \mathcal{O}^{*}(U_i)$
  such that $\widetilde{\varphi}|_{U_i \cap U_{j}}=\varphi$. As they are regular invertible functions, too,
  we can extend in exactly the same way the finite entries of the matrix
  $D_{j_0}$ along the chain $U_{j_0-1},U_{j_0-2},\ldots,U_{i+1}$ to any set
  $U_{i+1}$ for $i \in \{2,\ldots,j_0-1\}$. By abuse of notation we will denote this continuation of $D_{j_0}$
  as well by $D_{j_0}$.
  Now, we take $U_i' := U_i$ for all $i=1,\ldots,s$ and
  $$ M_{i,i+1}'(x) := \left\{ \begin{array}{ll} P_{j_0} \odot D_{j_0}(x) \odot M_{i,i+1}(x) \odot D_{j_0}(x)^{-1} \odot P_{j_0}^{-1}, & \text{if } i\in \{2,\ldots,j_0-1\}\\
  M_{i,i+1}(x), & \text{if } i\in \{j_0+1,\ldots,s\}.\end{array} \right.$$
  Moreover, we set $M_{12}'(x) := P_{j_0} \odot D_{j_0}(x) \odot D_{1}(x) \odot
  P_{1}$ and $M_{j_0,j_0+1}'(x) := E$. To check that the vector
  bundle $F'$ we obtain from this gluing data is isomorphic to $F$
  we apply lemma \ref{lemma-isomorphicbundles}: We set
  $$ E_i(x) := \left\{ \begin{array}{ll} D_{j_0}(x) \odot P_{j_0}, & \text{if } i \in \{2,\ldots,j_0\} \\
  E, & \text{else,}\end{array} \right.$$
  and get
  $$\begin{array}{rcl}
    (D_{j_0} \odot P_{j_0}) \odot (D_1 \odot P_1) &=& (D_{j_0} \odot P_{j_0} \odot D_1 \odot P_1) \odot E\\
    (D_{j_0} \odot P_{j_0}) \odot (D_2 \odot P_2) &=& (D_{j_0} \odot P_{j_0} \odot D_2 \odot P_2 \odot D_{j_0}^{-1} \odot P_{j_0}^{-1}) \odot (D_{j_0} \odot P_{j_0})\\
    \vdots & & \vdots \\
    E \odot (D_{j_0} \odot P_{j_0}) &=& E \odot (D_{j_0} \odot P_{j_0}).
  \end{array}$$
  This finishes our proof.
\end{proof}

To classify all vector bundles on our elliptic curve $X$ we give
now a non-redundant parametrization of all indecomposable vector
bundles on $X$. Arbitrary vector bundles are then just direct sums
of these building blocks.

\begin{theorem}[Vector bundles on elliptic curves] \label{thm-vecbundlesoverellcurves}
  Let $X:=X_2$ be the curve from \cite[example 5.5]{AR07}.
  Then the set of indecomposable vector bundles of rank $r$ and
  degree $d$ is in natural bijection with $\gcd(r,d) \cdot X$, i.e.
  with points of the curve $X$ stretched to $\gcd(r,d)$ times the original length.
\end{theorem}
\begin{proof}
  Let $\pi: F \rightarrow X$ be an indecomposable vector bundle of rank $r$ with
  open covering $U_1,\ldots,U_s$ and transition maps
  $M_{ij}$ according to definition \ref{def-vectorbundle}.
  Again, we may assume that all sets $U_i$ are connected, that for all $i,j$ the intersections $U_i \cap
  U_j$ are connected as well and that the sets $U_i$ are numbered
  consecutively. Moreover, by lemma \ref{lemma-needonlyonetransitionmap} we
  may assume that $M_{12}$ is the only non-trivial transition map.
  Let $M_{12}(x)= D(\varphi_1,\ldots,\varphi_r)(x) \odot A_\sigma =: D(x) \odot A_\sigma$
  for some regular invertible functions $\varphi_1,\ldots,\varphi_r
  \in \mathcal{O}^{*}(U_1 \cap U_2)$ and a permutation $\sigma \in
  S_r$. As $F$ is indecomposable $\sigma$ must by a single cycle.
  Hence there exists $\varrho \in S_r$ such that
  $\varrho\sigma\varrho^{-1}=(12 \ldots r)$. We will apply lemma \ref{lemma-isomorphicbundles}
  to show that we can replace $M_{12}(x)$ by $M_{12}'(x):= A_\varrho \odot D(x) \odot A_{\varrho^{-1}} \odot
  A_{(12 \ldots r)}$ without changing the isomorphism class of
  $F$: We set $E_i(x):=A_\varrho$ for all $x$ and all $i$ and
  obtain
  $$\begin{array}{rcl}
    A_\varrho \odot (D(x) \odot A_\sigma) &=& (A_\varrho \odot D(x) \odot A_{\varrho^{-1}} \odot A_{(12 \ldots r)}) \odot A_\varrho\\
    A_\varrho \odot E &=& E \odot A_\varrho\\
    \vdots & & \vdots \\
    A_\varrho \odot E &=& E \odot A_\varrho.
  \end{array}$$
  Hence we may assume that $\sigma=(12 \ldots r)$. Our next step
  is to apply lemma \ref{lemma-isomorphicbundles} to show that we may replace $D(x)=D(\varphi_1,\ldots,\varphi_r)$
  by $D'(x)=D(\varphi',0,\ldots,0)$ for some $\varphi' \in \mathcal{O}^{*}(U_1 \cap U_2)$
  without changing the isomorphism class of $F$. For $i=1,\ldots,r$ let
  $\alpha_i$ be the slope of $\varphi_i$ and let $L$ be the (lattice) length
  of our curve $X$. For $i=2,\ldots,r$ we set $\delta_i := \sum_{j=i}^r (j-i+1) \cdot \alpha_j$.
  Moreover, we define $\varphi' := \varphi_1 + \ldots + \varphi_r - \delta_2 L$.
  Note that if we are given a regular invertible
  function $\psi \in \mathcal{O}^{*}(U_i \cap U_{j})$ there is a unique
  regular invertible function $\widetilde{\psi} \in \mathcal{O}^{*}(U_i)$
  such that $\widetilde{\varphi}|_{U_i \cap U_{j}}=\varphi$. Hence we can
  extend our regular invertible functions $\varphi_1,\ldots,\varphi_r$ along the chain
  $U_2,U_3,\ldots,U_s,U_1$ to any of the sets $U_1,\ldots,U_s$. Note
  that on $U_1 \cap U_2$ the extension of $\varphi_i$ to $U_2$ and the extension of $\varphi_i$ to $U_1$
  differ exactly by $\alpha_i L$. We use these continuations to
  define the maps $E_i$:
  $$E_i(x) := D(\widetilde{\varphi_2}+\ldots+\widetilde{\varphi_r}-\delta_2 L, \widetilde{\varphi_3}+\ldots+\widetilde{\varphi_r}-\delta_3
  L,\ldots,\widetilde{\varphi_r}-\delta_r L,0),$$
  where for entries of $E_i$ the map $\widetilde{\varphi_j}$ denotes the continuation of $\varphi_j$
  to $U_i$. Hence we obtain on $U_1 \cap U_2$:
  {\small $$\begin{array}{rcl}
    && E_2 \odot M_{12}\\ &=& D(\widetilde{\varphi_2}+\ldots+\widetilde{\varphi_r}-\delta_2 L,\ldots,\widetilde{\varphi_r}-\delta_r L,0)
    \odot (D(\varphi_1,\ldots,\varphi_r) \odot A_\sigma)\\
    &=& D(\varphi_2+\ldots+\varphi_r-\delta_2 L,\ldots,\varphi_r-\delta_r L,0)
    \odot (D(\varphi_1,\ldots,\varphi_r) \odot A_\sigma)\\
    &=& D(\varphi_1+\ldots+\varphi_r-\delta_2 L,\varphi_2+\ldots+\varphi_r-\delta_3 L,\ldots,\varphi_{r-1}+\varphi_r-\delta_r
    L,\varphi_r) \odot A_\sigma
  \end{array}$$}
  and
  {\small $$\begin{array}{rcl}
    && M_{12}' \odot E_1\\ &=& (D(\varphi_1+\ldots+\varphi_r-\delta_2
    L,0,\ldots,0) \odot A_\sigma) \odot D(\widetilde{\varphi_2}+\ldots+\widetilde{\varphi_r}-\delta_2 L,\ldots,\widetilde{\varphi_r}-\delta_r L,0)\\
    &=& (D(\varphi_1+\ldots+\varphi_r-\delta_2
    L,0,\ldots,0) \odot A_\sigma) \odot D(\varphi_2+\ldots+\varphi_r-\delta_3 L,\ldots,\varphi_r-\delta_{r-1} L,0)\\
    &=& D(\varphi_1+\ldots+\varphi_r-\delta_2 L,\varphi_2+\ldots+\varphi_r-\delta_3 L,\ldots,\varphi_{r-1}+\varphi_r-\delta_r
    L,\varphi_r) \odot A_\sigma.
  \end{array}$$}The other conditions are trivially fulfilled as $E_i|_{U_i \cap
  U_{i+1}} = E_{i+1}|_{U_i \cap U_{i+1}}$ for all $i \neq 1$.
  Hence we may assume that $M_{12}(x)=D(x) \odot A_\sigma =
  D(\varphi',0,\ldots,0)(x) \odot A_{(12\ldots r)}$. As $F$ is a
  vector bundle of degree $d$ the affine linear map $\varphi'$
  must have slope $-d$. Thus, the transition map $M_{12}$ is determined by the isomorphism class
  of $F$ up to translations of $\varphi'$. To prove the claim it remains to
  show that two vector bundles $F$ and $F'$ as above with transition maps
  $M_{12}(x)=D(\varphi,0,\ldots,0)(x) \odot A_{(12\ldots r)}$ and
  $M_{12}'(x)={D(\varphi+c L,0,\ldots,0)(x) \odot A_{(12\ldots r)}}$
  are isomorphic if and only if $c$ is an integer multiple of
  $\gcd(r,d)$: By lemma \ref{lemma-isomorphicbundles} $F$ and $F'$
  are isomorphic if and only if for all $i=1,\ldots,s$ there exists a map $E_i: U_i
  \rightarrow G(r)$ such that for all $i$ the equation $E_{i+1}(x) \odot M_{i,i+1}(x) =
  M_{i,i+1}'(x) \odot E_i(x)$ holds for all $x \in U_i \cap
  U_{i+1}$. As $M_{i,i+1}$ is trivial for all $i \neq 1$ these equations imply $E_i|_{U_i \cap U_{i+1}} =
  E_{i+1}|_{U_i \cap U_{i+1}}$ for all $i \neq 1$.
  Hence $F$ and $F'$ are isomorphic if and only if there exist a permutation $\tau \in S_r$ and regular
  invertible functions $\psi_1,\ldots,\psi_r \in \mathcal{O}^{*}(U_1 \cap U_2)$
  with continuations $\widetilde{\psi_1},\ldots,\widetilde{\psi_r}$ to
  all sets $U_1,\ldots,U_s$ along the chain $U_2,U_3,\ldots,U_s,U_1$ such that
  {\small $$(D(\widetilde{\psi_1},\ldots,\widetilde{\psi_r}) \odot A_\tau) \odot
  (D(\varphi,0,\ldots,0) \odot A_\sigma) = (D(\varphi+cL,0,\ldots,0) \odot
  A_\sigma) \odot (D(\widetilde{\psi_1},\ldots,\widetilde{\psi_r}) \odot
  A_\tau)$$}holds on $U_1 \cap U_2$. In particular, the last equation
  implies $A_\tau \odot A_\sigma=A_\sigma \odot A_\tau$ and hence
  $\tau = \sigma^k$ for some $k \in \ZZ$. Thus $F$ and $F'$ are
  isomorphic if and only if there exist $k \in \ZZ$ and $\psi_1,\ldots,\psi_r$
  as above such that
  {\small $$D(\widetilde{\psi_1},\ldots,\widetilde{\psi_k},\widetilde{\psi_{k+1}}+\varphi,\widetilde{\psi_{k+2}},\ldots,\widetilde{\psi_r}) \odot A_{\sigma^{k+1}} =
  D(\varphi+cL+\widetilde{\psi_r},\widetilde{\psi_1},\ldots,\widetilde{\psi_{r-1}}) \odot
  A_{\sigma^{k+1}}.$$}
  Let $\alpha_i$ be the slope of $\psi_i$. Then on $U_1 \cap U_2$ the continuation of $\psi_i$ to $U_2$
  and the continuation of $\psi_i$ to $U_1$ differ exactly by $\alpha_i
  L$. Hence we obtain the system of equations
  $$\begin{array}{lll}
    \psi_1 &=& \varphi +cL+\psi_r+\alpha_r L\\
    \psi_2 &=& \psi_1 + \alpha_1 L\\
    \vdots & & \vdots\\
    \psi_k &=& \psi_{k-1} + \alpha_{k-1} L\\
    \psi_{k+1} + \varphi &=& \psi_k + \alpha_k L\\
    \psi_{k+2} &=& \psi_{k+1} + \alpha_{k+1} L\\
    \vdots & & \vdots\\
    \psi_r &=& \psi_{r-1} + \alpha_{r-1} L.
  \end{array}$$
  In particular, we can conclude that $\alpha_1=\ldots=\alpha_k$
  and $\alpha_{k+1}=\ldots=\alpha_r$. Hence $F$ and $F'$ are
  isomorphic if and only if there exist $\alpha_1, \alpha_r, k \in \ZZ$ such that
  $$-c = (r-k) \cdot \alpha_r + k \cdot \alpha_1 \text{ and } \alpha_1
  = -d + \alpha_r,$$ or equivalently if and only if there exist
  $\alpha_r, k \in \ZZ$ with $$-c = r \alpha_r - k \cdot d.$$
  This finishes the proof.
\end{proof}

\begin{remark}
  Note that the claim of theorem \ref{thm-vecbundlesoverellcurves}
  coincides with the equivalent result in ``classical'' algebraic geometry
  (see \cite[theorem 7]{A57}).
\end{remark}

I would like to thank my advisor Andreas Gathmann for numerous
helpful discussions.

\end{document}